\documentclass[a4paper, 11pt]{amsart}
\usepackage{amsmath}
\usepackage{amsthm}
\usepackage{amssymb}
\usepackage{xypic}
\usepackage[only,mapsfrom]{stmaryrd}
\usepackage{graphicx}

\newtheorem{theorem}{Theorem}[section]
\newtheorem{lemma}[theorem]{Lemma}

\newtheorem{corollary}[theorem]{Corollary}
\newtheorem{conjecture}[theorem]{Conjecture}
\newtheorem{question}[theorem]{Question}

\theoremstyle{definition}

\theoremstyle{remark}

\numberwithin{equation}{section}

\DeclareMathOperator{\sign}{sign}
\DeclareMathOperator{\td}{td}
\DeclareMathOperator{\wt}{wt}

\title{On circle actions with exactly three fixed points}

\author{Michael Wiemeler}
\address{
Mathematisches Institut\\ Universit\"at M\"unster\\Einsteinstrasse 62\\D-48149 M\"unster\\Germany
}
\email{wiemelerm@uni-muenster.de}



\date{6th August 2023}

\begin{document}
\begin{abstract}
  We study smooth, closed orientable \(S^1\)-manifolds \(M\)  with exactly \(3\) fixed points.
  We show that the dimension of \(M\) is of the form \(4\cdot 2^a\) or \(8\cdot(2^a+2^b)\) with \(a,b\geq 0\) and \(a\neq b\).
  Moreover, under the extra assumption that \(M\) is spin or unitary we determine exactly the dimension of \(M\).
\end{abstract}

\maketitle

\section{Introduction}

The study of symmetries of closed manifolds is an interesting subject with a long history.
One question in this study is the following:

\begin{question}
  \label{sec:introduction}
  Given a class \(\mathfrak{M}\) of closed smooth manifolds and positive integers \(n\) and \(k\), does there exist a manifold \(M\) in \(\mathfrak{M}\) of dimension \(n\) admitting a smooth \(S^1\)-action with exactly  \(k\) fixed points?
\end{question}


Related to Question~\ref{sec:introduction} is the following conjecture due to Kosniowski \cite{MR585666}.

\begin{conjecture}
Let \(M\) be a  closed unitary \(S^1\)-manifold
of dimension \(2n\) with only isolated fixed points. If \(M\) does not bound a unitary \(S^1\)-manifold equivariantly, then the number of isolated fixed points is greater than or equal
to a linear function \(f (n)\) of \(n\).
\end{conjecture}

Here a unitary manifold is a smooth manifold together with a complex structure \(J\) on the stable tangent bundle \(TM\oplus \mathbb{R}^k\).
In this paper we assume that all \(S^1\)-actions on unitary manifolds are compatible with the given complex structure \(J\), i.e. \(J\) commutes with the \(S^1\)-action on \(TM\oplus \mathbb{R}^k=TM\times \mathbb{R}^k\).
Here \(S^1\) acts on \(TM\) by the differential of the action on \(M\) and trivially on the \(\mathbb{R}^k\) factor.
Note that every unitary manifold is orientable.

This conjecture and the above question have been studied before for different classes of manifolds and actions including actions on symplectic manifolds and almost complex manifolds (see for example \cite{MR776689}, \cite{MR2818694}, \cite{MR2802578},  \cite{MR2916364}, \cite{MR3272779}, \cite{MR3230015}, \cite{MR3506787},  \cite{zbMATH06498086}, \cite{zbMATH06738992}, \cite{MR3649458}, \cite{MR4161918}, \cite{MR4050069}).

Here we study Question~\ref{sec:introduction} for the classes of smooth closed orientable, spin and unitary manifolds with \(k=3\) with smooth actions (and in the unitary case with actions preserving the unitary structure).
Here a spin manifold is an orientable manifold whose second Stiefel-Whitney class vanishes.
For the above classes the following restrictions on \(n\) are known, when \(k\leq 3\).

 \begin{enumerate}
 \item Since the fixed point set \(M^{S^1}\) of any \(S^1\)-manifold \(M\) is a submanifold of even codimension, \(S^1\)-manifolds with at least one but at most finitely many fixed points can only exist in even dimensions.
   \item There is no closed orientable \(S^1\)-manifold of positive dimension with exactly one fixed point (see for example \cite[Remark 3.2]{zbMATH06498086}).
   \item In all even dimensions spheres with certain linear actions provide examples of \(S^1\)-manifolds with exactly two fixed points. More generally, if a closed orientable \(S^1\)-manifold has exactly two fixed points, then the isotropy representations at the two fixed points are isomorphic (see \cite[Theorem 6]{MR585666}).
  \item For a closed oriented \(S^1\)-manifold \(M\) with isolated fixed points the signature \(\sign(M)\) of \(M\) can be computed as
    \[\sign(M)=\sign(M^{S^1})\equiv |M^{S^1}| \mod 2\]
    (see \cite[Corollary 6.23]{MR1150492}).
    It follows that \(\dim M\equiv 0 \mod 4\) if \(M\) has an odd number of fixed points.
    If, moreover, \(M\) is spin then \(\dim M\equiv 0 \mod 8\).
    This is because if \(\dim M\equiv 4\mod 8\) and \(M\) is spin, then the intersection form of \(M\) is even and therefore \(\sign(M)\) is divisible by eight (see \cite[p. 114-115]{MR1189136} or \cite{MR615511} for a stronger result).
\item Examples of orientable manifolds admitting a circle action with exactly three fixed points are the projective planes \(\mathbb{C} P^2\), \(\mathbb{H} P^2\) and \(\mathbb{O} P^2\) of dimensions \(4\), \(8\) and \(16\), respectively.
  Note here that \(\mathbb{C} P^2\) admits a unitary structure, while
  \(\mathbb{H}P^2\) and \(\mathbb{O}P^2\) are spin.
\end{enumerate}

Our main result is that in most dimensions orientable, spin or unitary \(S^1\)-manifolds with three fixed points cannot exist.

To be more precise we show the following:

\begin{theorem}
  \label{sec:ll}
  Let \(M\) be a smooth, closed \(S^1\)-manifold with exactly three fixed points and \(\dim M >0\).
  \begin{enumerate}
  \item If \(M\) is orientable, then \(\dim M=4 \cdot 2^a\) or \(\dim M=8\cdot(2^a+2^b)\) with \(a,b\in \mathbb{Z}_{\geq 0}\), \(a\neq b\).
  \item If \(M\) is spin, then \(\dim M\in \{8,16\}\).
  \item If \(M\) is unitary, then \(\dim M=4\).
  \end{enumerate}
\end{theorem}

Note that non-existence of \(12\)-dimensional oriented \(S^1\)-manifolds with exactly \(3\) fixed points has been shown recently by Jang \cite{jang21:_circl}.
Note, moreover, that he \cite{zbMATH06738992} has also shown that an almost complex \(S^1\)-manifold with three fixed points has dimension four.
His arguments were based on a detailed analysis of the possible isotropy representations at the fixed points.
He used a similar argument before for the case of symplectic manifolds \cite{MR3272779}.

By combining Theorem~\ref{sec:ll} with Theorem 5 of \cite{MR585666} or Corollary 1.5 of \cite{MR2916364} we get

\begin{corollary}
  Let \(M\) be a closed unitary \(S^1\)-manifold of dimension greater than \(6\) that does not bound equivariantly. Then there are at least \(4\) fixed points in \(M\).
\end{corollary}

A \(S^1\)-action on a manifold \(M\) is called equivariantly formal if
\[\sum_{i\geq 0} b_i(M)=\sum_{i\geq 0} b_i(M^{S^1}),\]
where \(b_i(M)\) denotes the \(i\)-th Betti number of \(M\).
Moreover, for every \(S^1\)-action on a manifold \(M\), one has
\[\chi(M)=\chi(M^{S^1}),\]
where \(\chi(M)=\sum_{i\geq 0} (-1)^i b_i(M)\) denotes the Euler characteristic of \(M\).
Therefore a circle action on an orientable closed manifold \(M\) with exactly three fixed points is equivariantly formal if and only if \(M\) is a rational projective plane, i.e. if \(H^*(M;\mathbb{Q})=0\) for \(*\neq 0, \dim M, \frac{1}{2} \dim M\) and \(H^*(M;\mathbb{Q})=\mathbb{Q}\) for \(*= 0,\dim M, \frac{1}{2} \dim M\).

Hence one might guess that there are some similarities between those dimensions which support a rational projective plane and those dimensions which support an orientable \(S^1\)-manifold with three fixed points.
The former have been studied in \cite{MR3158765}, \cite{MR3500384},  \cite{MR3872943}, \cite{MR3999512}, \cite{hu21:_almos_betti} and \cite{su22:_almos_betti}.
One of the main tools in this study was the Hirzebruch signature theorem which expresses the signature of a manifold as a linear combination of its Pontrjagin numbers.
This was combined with the fact that a rational projective plane has signature one and can have at most two non-zero Pontrjagin numbers to deduce
 dimension restrictions for rational projective planes.

Similarly as a first major step in our study of orientable \(S^1\)-manifolds \(M\) with three fixed points we show that all except two of their Pontrjagin numbers vanish.
Moreover, we also show that the signature of \(M\) is one.
In this step we use the Atiyah--Bott--Berline--Vergne localization formula.
The arguments used in this step are similar to some arguments used before, for example in \cite{MR2802578}, \cite{zbMATH06194629}, \cite{MR2818694} and \cite{zbMATH06144281}.
To get from there to Theorem \ref{sec:ll} is then almost the same as in the study of rational projective planes.
We use arguments from \cite{MR3158765} and \cite{MR3500384} for the orientable case, from \cite{MR3999512} for the spin case and from \cite{hu21:_almos_betti} for the unitary case.

This paper is structured as follows.
In Section \ref{sec:preliminaries} we collect some basic facts about equivariant cohomology and the Hirzebruch signature theorem.
Then in Section~\ref{sec:proof-theorem-1} we prove Theorem~\ref{sec:ll}.

I would like to thank Anand Dessai for comments on an earlier version of this paper.

The research for this paper was funded by the Deutsche Forschungsgemeinschaft (DFG, German Research Foundation) under Germany's Excellence Strategy EXC 2044 –390685587, Mathematics M\"unster: Dynamics–Geometry–Structure and through CRC1442 Geometry: Deformations and Rigidity at University of M\"unster.

\section{Preliminaries}
\label{sec:preliminaries}

\subsection{Equivariant cohomology}

In this section we recall some facts about equivariant cohomology. For an introduction to this subject see \cite[Chapters 5 and 6]{MR1150492}.

Let \(M\) be an oriented \(S^1\)-manifold.
Then the equivariant cohomology of \(M\) is defined as
\[H_{S^1}^*(M;\mathbb{Q})=H^*(M_{S^1};\mathbb{Q}),\]
where \(M_{S^1}=ES^1\times_{S^1} M\) is the Borel construction of \(M\) and \(ES^1\rightarrow BS^1\) is the universal principal \(S^1\)-bundle.

Hence the inclusion of the fiber in the fibration \(M\rightarrow M_{S^1}\rightarrow BS^1\) induces a map \(\iota^*:H^*_{S^1}(M;\mathbb{Q})\rightarrow H^*(M;\mathbb{Q})\).
We say that a class \(\tilde{c}\in H^i_{S^1}(M;\mathbb{Q})\) is a lift of a class \(c\in H^i(M)\), \(i\geq 0\), if \(\iota^*(\tilde{c})=c\).

Moreover, the projection \(M_{S^1}\rightarrow BS^1\) induces a \(H^*(BS^1;\mathbb{Q})\)-algebra structure on \(H^*_{S^1}(M;\mathbb{Q})\).
Note that \(H^*(BS^1;\mathbb{Q})=\mathbb{Q}[t]\), where \(t\in H^2(BS^1;\mathbb{Q})\) is the Euler class of the universal \(S^1\)-bundle \(ES^1\rightarrow BS^1\).
Since \(\iota^*(t^i)=0\) for all \(i>0\), lifts of cohomology classes of degree \(2i\) are never unique.
If \(\tilde{c}\) is a lift of \(c\in H^{2i}(M;\mathbb{Q})\) then for every \(a\in \mathbb{Q}\), \(\tilde{c}+at^i\) is also a lift of \(c\).

If \(V\rightarrow M\) is an equivariant vector bundle over \(M\) then \(V_{S^1}\rightarrow M_{S^1}\) is a vector bundle over \(M_{S^1}\).
Hence one can define the equivariant characteristic classes of \(V\) as the characteristic classes of \(V_{S^1}\).
These are lifts of the ordinary characteristic classes of \(V\).

Next assume that the \(S^1\)-action on \(M\) has isolated fixed points and \(\dim M=2n\).

Then by \cite[Corollary 6.13]{MR1150492} we have for every \(\tilde{c}\in H^{2n}_{S^1}(M;\mathbb{Q})\),
\begin{equation}
  \label{eq:1}
\langle\iota^*(\tilde{c}),[M]\rangle=\sum_{q\in M^{S^1}} \frac{\tilde{c}|_q}{e(q)}\epsilon(q),
\end{equation}

where for a fixed point \(q\in M^{S^1}\),
\begin{itemize}
\item \(\tilde{c}|_q\) is the restriction of \(\tilde{c}\) to \(H^{2n}_{S^1}(q;\mathbb{Q})\cong H^{2n}(BS^1;\mathbb{Q})=\mathbb{Q}t^n\).

\item \(e(q)=t^n\prod_{i=1}^{n}\mu_i\) where the \(\mu_i\) are the weights of the \(S^1\)-representation \(T_qM\cong\bigoplus_{i=1}^{n} V_{\mu_i}\). Here \(V_{\mu_i}\) is \(\mathbb{C}\) with the \(S^1\)-action given by multiplication with \(g^{\mu_i}\) for \(g\in S^1\). We choose all \(\mu_i>0\).
\item \(\epsilon(q)=1\) if the orientation on \(T_qM\) agrees with the orientation induced by the complex structure on \(\bigoplus_{i=1}^{n} V_{\mu_i}\).
  Otherwise \(\epsilon(q)=-1\).
\end{itemize}

Note that, by \cite[Corollary 6.23]{MR1150492} with these choices we also have
\begin{equation}\label{eq:2}\sign(M)=\sum_{q\in M^{S^1}}\epsilon(q).\end{equation}

\subsection{The Hirzebruch signature theorem}

Here we recall some facts about the Hirzebruch signature theorem.
For a general introduction to this subject see \cite[Chapter 19]{MR0440554}.

Let \(k\) be a non-negative integer.
We say that a finite sequence \(I=(i_1,\dots,i_l)\) with \(0< i_1\leq i_2\leq\dots\leq i_l\) and \(\sum_{j=1}^l i_j=k\) is a partition of \(k\).
If \(M\) is an oriented manifold of dimension \(4k\) then we write \(p_I(M)=\prod_{j=1}^l p_{i_j}(M)\) for the partition \(I=(i_1,\dots,i_l)\) of \(k\). Here \(p_i(M)\in H^{4i}(M;\mathbb{Q})\) denotes the \(i\)-th Pontrjagin class of \(M\).
We call the \(p_I[M]=\langle p_I(M),[M]\rangle\) where \(I\) runs through all partitions of \(k\) the Pontrjagin numbers of \(M\).
Note that the Pontrjagin numbers of an oriented manifold are integers.

Similarly if \(M\) is a unitary manifold of dimension \(2k\) we write \(c_I(M)=\prod_{j=1}^l c_{i_j}(M)\) where \(c_i(M)\in H^{2i}(M;\mathbb{Q})\) denotes the \(i\)-th Chern class of \(M\).
We call the \(c_I[M]=\langle c_I(M),[M]\rangle\) where \(I\) runs through all partitions of \(k\) the Chern numbers of \(M\).
Note that the Chern numbers of an unitary manifold are integers.

According to the Hirzebruch signature theorem \cite[Theorem 19.4]{MR0440554}  the signature of a \(4k\)-dimensional oriented manifold \(M\) is equal to a linear combination of Pontrjagin numbers of \(M\),

\[\sign(M)=\sum_{I} s_Ip_I[M].\]

Here the sum runs over all partitions of \(k\) and the \(s_I\) are certain polynomials with rational coefficients in the Bernoulli numbers \(B_i\in \mathbb{Q}\).
The Bernoulli numbers \(B_i\) are defined by the equation
\[\frac{s}{e^s-1}=\sum_{i=0}^\infty\frac{B_i}{i!}s^i\]
for \(s\in \mathbb{C}\), \(|s|<2\pi\).
In particular (see \cite[Exercises 19-B and 19-C]{MR0440554}),
\begin{align*}
  s_k&=2^{2k}\left(2^{2k-1}-1\right)\frac{|B_{2k}|}{(2k)!}.
\end{align*}
Moreover, if \(k\) is even then
\[s_{k/2,k/2}=\frac{1}{2}\left(s_{k/2}^2-s_k\right).\]

Similarly if \(M\) is unitary of dimension \(2k\), one can write the Todd genus \(\td(M)\) of \(M\) as

\[\td(M)=\sum_I t_Ic_I[M],\]

where again the sum runs over all partitions of \(k\) and the \(t_I\) are again certain polynomials with rational coefficients in the Bernoulli numbers (see \cite[Exercise 19-A]{MR0440554}).
In particular,
\begin{align*}
  t_k&=(-1)^k\frac{B_{k}}{k!}.
\end{align*}
Moreover, if \(k\) is even then
\[t_{k/2,k/2}=\frac{1}{2}\left(t_{k/2}^2-t_k\right).\]

Note that the Todd genus of every unitary manifold is an integer.

\section{The proof of Theorem \ref{sec:ll}}
\label{sec:proof-theorem-1}

In this section we prove Theorem~\ref{sec:ll}.
The main technical input is the following lemma.
Its proof uses some arguments which are similar to arguments used in \cite{MR2802578}, \cite{zbMATH06194629}, \cite{MR2818694} and \cite{zbMATH06144281}.

\begin{lemma}
  \label{sec:proof-theorem}
  Let \(M\) be a closed \(S^1\)-manifold of dimension \(4k>0\) with exactly three fixed points. Then the following holds:
  \begin{enumerate}
  \item If \(M\) is oriented, then all Pontrjagin numbers of \(M\) except maybe \(p_{k/2,k/2}[M]\) and \(p_k[M]\) vanish. Moreover, \(\sign(M)=\pm 1\).
  \item If \(M\) is unitary then all Chern numbers of \(M\) except maybe \(c_{k,k}[M]\) and \(c_{2k}[M]\) vanish.
    Moreover, \(c_{2k}[M]\in \{\pm 1,\pm 3\}\).
  \end{enumerate}
\end{lemma}
\begin{proof}
  We only prove the claim in the unitary case. The proof in the oriented case goes along the same lines. One just has to replace Chern classes by Pontrjagin classes everywhere.

  Recall that in the case of unitary \(S^1\)-manifolds we assume that the \(S^1\)-action is compatible with the stable almost complex structure.
  Therefore we can choose lifts \(\tilde{c}_i\in H^{2i}_{S^1}(M;\mathbb{Q})\) of the Chern classes \(c_i=c_i(M)\in H^{2i}(M;\mathbb{Q})\).

  Let \(0<i< k\) and \(l=2k-2i> 0\).
  Define an equivalence relation \(\sim\) on \(M^{S^1}\) by \(q\sim q'\) if and only if \(\tilde{c}_i|_q=\tilde{c}_i|_{q'}\in H^{2i}(BS^1;\mathbb{Q})\). 
  Denote by \(A_1,\dots,A_m\) the equivalence classes of the relation and let \(a_n\in \mathbb{Q}\) such that \(a_nt^i=\tilde{c}_i|_q\) for \(q\in A_n\).

  Then for any \(a\in \mathbb{Q}\) we have, by Equation~(\ref{eq:1}),
  \begin{align*}
    0&=\langle \iota^*((\tilde{c}_i+at^i)^2t^l),[M]\rangle\\&=\sum_{q\in M^{S^1}}\frac{(\tilde{c}_i|_q + a t^i)^2 t^l}{e(q)}\epsilon(q)\\
                               &=\sum_{n=1}^m(a_n + a)^2 \sum_{q\in A_n}\frac{t^{2k}}{e(q)}\epsilon(q)\\
    &=\sum_{j=0}^2\binom{2}{j} a^{2-j}\sum_{n=1}^m a_n^j \sum_{q\in A_n}\frac{t^{2k}}{e(q)}\epsilon(q).
  \end{align*}
  
  Since the left hand side is independent of \(a\), the \(a_n\) are pairwise distinct and \(m\leq 3\),  we find, by comparing coefficients in \(a\),
  \[\sum_{q\in A_n}\frac{t^{2k}}{e(q)}\epsilon(q)=0,\]
  for all \(n=1,\dots,m\) because the Vandermond determinant
  \[\det (a_n^{j-1})_{1\leq j,n\leq m}=\prod_{n_2<n_1}(a_{n_1}-a_{n_2})\]
  is non-zero
  (see \cite[p. 149]{zbMATH05031717}). In particular, \(|A_n|\geq 2\) for all \(n\), because \(\frac{t^{2k}}{e(q)}\epsilon(q)\neq 0\) for all \(q\in M^{S^1}\).
  But since \(|M^{S^1}|=3\) this implies \(m=1\), i.e. \(\tilde{c}_i|_q\) is independent of \(q\in M^{S^1}\).

  Next consider a partition \(I\) of \(2k-i\).
  Then, for every \(a\in \mathbb{Q}\), we have, by Equation~(\ref{eq:1}),
 \begin{align*}
    c_{i,I}[M]&=\sum_{q\in M^{S^1}}\frac{(\tilde{c}_i|_q + a t^i)\tilde{c}_I|_q}{e(q)}\epsilon(q)\\
    &=(a_1 + a) \sum_{q\in M^{S^1}}\frac{\tilde{c}_I|_qt^{i}}{e(q)}\epsilon(q).
 \end{align*}
 
 Therefore by putting \(a=-a_1\), \(c_{i,I}[M]=0\) follows.
 
  Hence the only non-trivial Chern numbers can be \(c_{k,k}[M]\) and \(c_{2k}[M]\).
  
  The claim about \(c_{2k}[M]\) follows from Equation~(\ref{eq:1}) because we can choose \(\tilde{c}_{2k}\) such that \(\tilde{c}_{2k}|_q=\pm e(q)\) for all \(q\in M^{S^1}\).
  The claim about the signature follows from Equations~(\ref{eq:1}) and~(\ref{eq:2}) because, by
  \[0=\langle\iota^*(t^{2k}),[M]\rangle=\sum_{q\in M^{S^1}} \frac{t^{2k}}{e(q)}\epsilon(q),\]
    not all \(\epsilon(q)\) can have the same sign.
\end{proof}

\subsection{Oriented manifolds}

In this subsection we prove the claim about oriented manifolds in Theorem~\ref{sec:ll}.
In this case the claim basically follows from the arguments in the
proof of Lemma 3.2 of \cite{MR3158765} and the
proof of Theorem A in \cite[Section 2]{MR3500384} in combination with Lemma~\ref{sec:proof-theorem}.
Note, however, that the cited proofs are formulated with the application to rational projective planes in mind.
Therefore for the convenience of the reader we also include the argument here.

Let \(M\) be an oriented \(S^1\)-manifold of positive dimension with exactly three fixed points.
Then by Lemma~\ref{sec:proof-theorem} and the Hirzebruch signature theorem  we have
\[s_{2k+1}p_{2k+1}[M]=\sign(M)=\pm 1\]
if \(\dim M= 8k +4\), and 
\[s_{k,k}p_{k,k}[M]+s_{2k}p_{2k}[M]=\sign(M)=\pm 1\]
if \(\dim M=8k\).

As shown in the proof of Proposition 2.1 of \cite[p. 218]{MR3500384}, we have \(\nu_2(s_k)=\wt(k)-1\) and \(\nu_2(s_{k,k})\geq \wt(k)-2\), where:

\begin{itemize}
\item The Hamming weight function \(\wt : \mathbb{Z}\rightarrow  \mathbb{Z}_{\geq 0}\) sends an integer \(x\) to the number of \(1\)’s in the
binary expansion of \(|x|\).
\item The \(2\)-adic valuation, \(\nu_2:\mathbb{Z}\rightarrow \mathbb{Z}_{\geq 0}\cup\{\infty\}\), sends \(x\neq 0\) to the largest integer \(\nu\) so that \(2^\nu\) divides \(x\).
  Moreover, we set \(\nu_2(0)=\infty\).
\item  The extended \(2\)-adic valuation \(\nu_2 : \mathbb{Q}\rightarrow \mathbb{Z}\cup\{\infty\}\) is defined as \(\nu_2\left( \frac{x}{y} \right) = \nu_2 (x) - \nu_2(y)\).
\end{itemize}

Hence we get
\begin{align*}0&=\nu_2(\sign(M))=\nu_2(s_{2k+1})+\nu_2(p_{2k+1}[M])\\
               &\geq \nu_2(s_{2k+1})=\wt(2k+1)-1\end{align*}
if \(\dim M=8k+4\).
Therefore \(k=0\) follows in this case.

In the case \(\dim M=8k\) we have
\begin{align*}0&=\nu_2(\sign(M))\\
  &\geq\min(\nu_2(s_{2k})+\nu_2(p_{2k}[M]),\nu_2(s_{k,k})+\nu_2(p_{k,k}[M]))\\
  &\geq\wt(k)-2.\end{align*}
Hence \(k= 2^a\) or \(k=2^a+2^b\) with \(a,b\geq 0\), \(a\neq b\) follows in this case.
This proves the claim in Theorem~\ref{sec:ll} about oriented manifolds.

\subsection{Spin manifolds}

In this subsection we prove the claim about spin manifolds in Theorem~\ref{sec:ll}.

As in the proof of Corollary 19 in \cite{MR3999512} one sees that if there is a spin manifold of dimension \(4k\) with

\begin{enumerate}
\item the only possibly non-zero Pontrjagin numbers of \(M\) are \(p_{k/2,k/2}[M]\) and \(p_{k}[M]\), and
\item \(\sign(M)=\pm 1\)
\end{enumerate}
then \(k=2\) or \(k=4\).

In combination with our Lemma \ref{sec:proof-theorem} the claim about spin manifolds in Theorem~\ref{sec:ll} follows.

The proof of \cite[Corollary 19]{MR3999512} is based on
\cite[Theorem 18]{MR3999512}.
Note that while this theorem is stated in loc. cit. only for simply connected spin manifolds it also holds for non-simply connected spin manifolds.
This is because by Assertion 1 in \cite[p. 202]{zbMATH03189976} every spin bordism class in dimension at least six can be represented by a two-connected manifold and Pontrjagin numbers as well as the signature are spin bordism invariants.

\subsection{Unitary manifolds}

In this section we prove the claim about unitary \(S^1\)-manifolds in Theorem~\ref{sec:ll}.
The proof uses arguments from \cite{hu21:_almos_betti} for dimensions at least \(16\) and some arguments which are specific to dimension \(8\).

It follows from the proof of Proposition 3.2 in \cite{hu21:_almos_betti} that for \(k\geq 2\) there is no \(8k\)-dimensional unitary manifold \(M\) such that
\begin{enumerate}
\item all Chern numbers except maybe \(c_{4k}[M]\) and \(c_{2k,2k}[M]\) vanish,
\item  \(\sign(M)\neq 0\) and
\item  \(\max\{|\sign(M)|,|c_{4k}[M]|\}\leq 8\).
\end{enumerate}

Note, here, that the proof in loc. cit. is only formulated for almost complex manifolds.
However, by replacing the Euler characteristic by the characteristic number \(c_{4k}[M]\) in all arguments in loc. cit. (proof of Proposition 3.2 and Section 2) one gets the above statement.
Note here that for almost complex manifolds \(M\) we have \(c_{4k}[M]=\chi(M)\).
Alternatively, one can use that every unitary bordism class can be represented by a complex manifold (see \cite{zbMATH03214366}) to deduce the above statement from the corresponding statement for almost complex manifolds.
Hence, for \(k\geq 2\) there is no \(8k\)-dimensional unitary manifold satisfying the constraints of Lemma~\ref{sec:proof-theorem}.

So we only have to exclude the possibility of the existence of an \(8\)-dimensional unitary \(S^1\)-manifold with exactly three fixed points.
Assume there is such a manifold \(M\).
After changing the orientation of \(M\) if necessary we can assume \(\sign(M)>0\).

The Pontrjagin classes of the unitary manifold \(M\) can be computed from the Chern classes of the manifold as follows:
\[p_k(M)=\sum_{i=0}^{2k}(-1)^{i+k}c_i(M)c_{2k-i}(M).\]
So, by Lemma~\ref{sec:proof-theorem}, we have
\begin{align*}
  p_{1,1}[M]&=4c_{2,2}[M],&p_{2}[M]&=2c_{4}[M]+c_{2,2}[M].
\end{align*}

By combining this with the Hirzebruch signature theorem we get
\begin{equation}
  \label{eq:3}\sign(M)=\left(2s_1^2-s_{2}\right)c_{2,2}[M]+ 2 s_{2} c_{4}[M],\end{equation}
with 
\(s_1=\frac{1}{3}\) and \(s_2=\frac{7}{45}\).

We also have
\begin{equation}\label{eq:4}\td(M)=\frac{1}{2}\left(t_{2}^2-t_{4}\right)c_{2,2}[M] + t_{4}c_{4}[M]\end{equation}
with \(t_2=\frac{1}{12}\) and \(t_4=-\frac{1}{720}\).

Since the Todd genus of an unitary manifold as well as its Chern numbers are integers,
Equations~(\ref{eq:3}) and (\ref{eq:4}) have only one solution satisfying the constraints from Lemma~\ref{sec:proof-theorem}, namely
\begin{align*}
  \sign(M)&=1&c_{4}[M]&=3&c_{2,2}[M]&=1&\td(M)&=0.
\end{align*}

To exclude this solution we consider the Hirzebruch \(\chi_y\)-genus of \(M\).
By the Kosniowski formula (see \cite[Section 3]{zbMATH03378250}, \cite{zbMATH03307319}), we have
\[\chi_y(M)=\sum_{q\in M^{S^1}} (-y)^{d_q}\epsilon'(q)=\sum_{q\in M^{S^1}} (-y)^{4-d_q}\epsilon'(q).\]
Here, for \(q\in M^{S^1}\), \(d_q=\dim_{\mathbb{C}}\bigoplus_{i,\mu_i'>0} V_{\mu_i'}\), where \(T_qM=\bigoplus_{i=1}^4V_{\mu_i'}\), \(\mu_i'\in \mathbb{Z}\), is the decomposition of the complex \(S^1\)-representation \(T_qM\) into irreducible subrepresentations. Moreover, \(\epsilon'(q)\in \{\pm 1\}\).

  Since \(\chi_{-1}(M)=c_4[M]=3\), we must have \(\epsilon'(q)=1\) for all \(q\in M^{S^1}\).
  Because \(\chi_1(M)=\sign(M)=1\), there must be exactly one \(q\in M^{S^1}\) such that \(d_q\) is odd.
  There must also be one \(q'\in M^{S^1}\) such that \(d_{q'}=4-d_q\).
  But then \(d_{q'}\) is also odd and not equal to \(d_q\).
  So \(q'\neq q\), a contradiction.

  Hence we have proved the claim about unitary manifolds in Theorem~\ref{sec:ll}.


\bibliography{refs}{}
\bibliographystyle{alpha}
\end{document}